\documentclass{amsproc}
\newtheorem{theorem}{Theorem}[section]
\newtheorem{lemma}[theorem]{Lemma}
\newtheorem{corollary}[theorem]{Corollary}
\theoremstyle{definition}
\newtheorem{definition}[theorem]{Definition}
\newtheorem{example}[theorem]{Example}

\theoremstyle{remark}
\newtheorem{remark}[theorem]{Remark}
\numberwithin{equation}{section}

\usepackage{graphicx}
\usepackage{amssymb}
\usepackage{epstopdf}
\usepackage{epsfig}
\usepackage{psfrag}
 \newcommand{\res}{\operatorname{res}}

\begin{document}

 \title[Minkowski Measurability and Exact Tube Formulas for $p$-Adic Strings] {Minkowski Measurability and Exact  Fractal Tube Formulas for $p$-Adic Self-Similar  Strings}

\author{Michel L. Lapidus}
\address{Department of Mathematics, University of California, Riverside, CA 92521-0135}
\email{lapidus@math.ucr.edu}
\thanks{The research of the first author (MLL) was partially supported by the US National Science Foundation under the grant DMS-0707524 and DMS-1107750, as well as by the Institut des Hautes Etudes Scientifiques (IHES) where the first author was a visiting professor in the Spring of 2012 while this paper was completed.}
\author {L\~u' H\`ung}
\address{Department of Mathematics, Hawai`i Pacific University, Honolulu, HI 96813-2785}
\email{hlu@hpu.edu}
\thanks{The research of the second author (LH) was partially supported by the Trustees' Scholarly Endeavor Program at Hawai`i Pacific University.}
\author {Machiel van Frankenhuijsen}
\address{Department of Mathematics, Utah Valley University, Orem, UT 84058-5999}
\email{ vanframa@uvu.edu}
\thanks{The research of the third author (MvF) was partially supported by the Georg-August-Universit\"at G\"ottingen and the College of Science and Health of Utah Valley University.}


\date{\today}


\begin{abstract}
The theory of $p$-adic fractal strings and their complex dimensions was developed by the first two authors in \cite{LapLu1, LapLu2, LapLu4}, particularly in the self-similar case, in parallel with its archimedean (or real) counterpart developed by the first and third author in \cite{L-vF2}. Using the fractal tube formula obtained by the authors for $p$-adic fractal strings in \cite{LapLu3},
we present here an exact  volume formula for the tubular neighborhood of a $p$-adic self-similar fractal  string $\mathcal{L}_p$, expressed in terms of the underlying complex dimensions.  
The periodic structure of the complex dimensions allows one to obtain a very concrete form for the resulting fractal tube formula. 
 Moreover, we derive and use a truncated version of this fractal tube formula  in order to  show that 
$\mathcal L_p$ is not Minkowski measurable and  obtain  an explicit expression for its average Minkowski content. 
The general theory is illustrated by two simple examples, the 3-adic Cantor string and the 2-adic Fibonacci strings, which are nonarchimedean analogs (introduced in \cite{LapLu1, LapLu2}) of the real Cantor and Fibonacci strings studied in \cite{L-vF2}. 
\end{abstract}

\maketitle

\tableofcontents

\begin{quote}
{\em
Nature is an infinite sphere of which the center is everywhere and the circumference nowhere.}
\hspace{\stretch{1}} Blaise Pascal %(1623-1662)
\end{quote}

\section{Introduction} 

In this paper, we present and use   the explicit tube formulas obtained in \cite{LapLu3}, for general $p$-adic fractal strings, in order to derive exact fractal tube formulas for $p$-adic self-similar fractal strings. The general results are illustrated in the case of suitable nonarchimedean analogs of the Cantor and the Fibonacci strings. 
Some particular attention is devoted to the 3-adic (or nonarchimedean) Cantor string (introduced and studied in \cite{LapLu1}, an appropriate counterpart of the archimedean Cantor string, whose `metric boundary' is the  3-adic Cantor set \cite{LapLu1}), a suitable $p$-adic analog of the classic ternary Cantor set. We also  derive  an explicit expression for the average Minkowski content of a $p$-adic self-similar string and the `boundary' of the associated nonarchimedean self-similar set.

We note that $p$-adic (or nonarchimedean) analysis has been used in various areas of mathematics (such as functional analysis and operator theory, representation theory, number theory and arithmetic geometry), as well as (sometimes more speculatively) of mathematical and theoretical physics (such as quantum mechanics, relativity theory, quantum field theory, statistical and condensed matter physics, string theory and cosmology); see, e.g., \cite{Drag, Dykkv, Ulam,  Khrennikov, RTV, VVZ, Vol} and the relevant references therein. 
In particular, ultrametric structures have been shown to be very useful tools to study spin glasses in condensed matter physics; see \cite{RTV} for a comprehensive survey on this and related topics. We also point out the more recent review article \cite{Dykkv} which discusses a variety of potential applications of $p$-adic analysis in mathematical physics and biology. 
 Furthermore, several physicists and mathematical physicists have suggested that the small scale structure of spacetime may be fractal; see, e.g., \cite{GibHaw, HawIs, Lap2, Not, WheFo}.
 In addition, it has been suggested (in \cite{Vol} for example) that seemingly abstract objects such as nonarchimedean fields (including the field of $p$-adic numbers) can be helpful in order to understand the geometry of spacetime at sub-Planckian scales.

Finally, we note that $p$-adic fractal strings (and their possible quantized analogs) may be helpful to obtain an appropriate adelic counterpart of ordinary (real) fractal strings, along with their quantization (called fractal membranes), as introduced in \cite{Lap2}.

\section[$p$-Adic Numbers]{$p$-Adic Numbers}
\label{S:p-adic numbers}

Given a fixed prime number $p$,
 any nonzero rational number $x$ can be written as $x=p^v\cdot a/b$,
 for integers $a$ and $b$ and a unique exponent $v\in \mathbb Z$ such that $p$ does not divide $a$ or $b$.
The {\em $p$-adic absolute value\/} is the function $|\cdot|_p\colon\mathbb Q \rightarrow [0,\infty)$ given by $|x|_p=p^{-v}$ and $|0|_p=0$.
It satisfies the {\em strong triangle inequality\/}:
 for every $x,y\in \mathbb Q$,
$$
|x+y|_p\leq \max\{|x|_p, |y|_p\}.
$$
Relative to the $p$-adic absolute value,
 $\mathbb Q$ does not satisfy the archimedean property because for each $x\in \mathbb Q$, $|nx|_p$ will never exceed $|x|_p$ for any $n\in \mathbb N$. The completion of $\mathbb Q$ with respect to  $|\cdot|_p$ is the field of $p$-adic numbers $\mathbb Q_p$.
More concretely,
every $z\in \mathbb Q_p$ has a unique representation
$$
z=a_{v}p^{v} + \cdots + a_0 + a_1p + a_2 p^2 + \cdots,
$$
for some $v\in \mathbb Z$ and  $a_j \in \{0, 1, \dots, p-1\}$ for all $j\geq v$ and $a_v\neq0$.
An important subspace of $\mathbb Q_p$ is the unit ball,
 $\mathbb Z_p=\{x\in \mathbb Q_p\colon |x|_p\leq 1\}$,
 which can also be represented as follows:
\[
\mathbb Z_p=\left\{ a_0 + a_1p + a_2 p^2 + \cdots\colon a_j \in \{0, 1, \dots, p-1\} \text{ for all } j\geq 0 \right\}.
\]
Using this $p$-adic expansion,
 one sees that
\begin{equation}
\mathbb Z_p=\bigcup_{a=0}^{p-1} (a+p\mathbb Z_p),
\label{decomposition}
\end{equation}
where $a+p\mathbb Z_p=\{y\in \mathbb Q_p\colon |y-a|_p\le 1/p\}$.
Note that $\mathbb Z_p$ is compact and thus complete.
Also,
 $\mathbb Q_p$ is a locally compact group,
 and hence admits a unique translation invariant Haar measure $\mu_H,$
normalized so that $\mu_H(\mathbb Z_p)=1$.
In particular,
 $ \mu_H(a+p^n\mathbb Z_p)=p^{-n}$ for every $n\in \mathbb Z$. For general references on $p$-adic analysis, we point out, e.g.,
\cite{Kob,Rob,Sch, Ser}. 
 
 Here and thereafter, we use the following notation: $\mathbb N=\{0, 1,2, \ldots\}$, $\mathbb N^*=\{1, 2, 3, \ldots\}$ and $\mathbb Z=\{0, \pm 1, \pm 2, \ldots\}$.

\section{$p$-Adic Fractal Strings}\label{pfs}
 Let  $ \Omega$ be a bounded open subset of $\mathbb{Q}_p$.  
 Then it can be decomposed into a countable union of disjoint open balls\footnote
{We shall often call a $p$-adic ball an \emph{interval}.  By `ball' here, we mean a metrically closed and hence, topologically open (and closed) ball.}
 with radius $p^{-n_j}$ centered at $a_j\in \mathbb Q_p$, 
 \[ a_j + p^{n_j} \mathbb{Z}_p=B(a_j, p^{-n_j})=\{x\in \mathbb{Q}_p ~|~ |x- a_j|_p \le p^{-n_j}\},\]
  where $n_j \in \mathbb{Z}$ and $j\in \mathbb N^{*}$. 
  There may be many different such decompositions since each ball can always be decomposed into smaller disjoint balls \cite{Kob}; see Equation (\ref{decomposition}).
 However, there is  a canonical decomposition of $\Omega$ into disjoint balls with respect to a suitable equivalence relation, as we now explain.
  
  \begin{definition}\label{relation}
 Let $U$ be an open subset of $\mathbb Q_p$. Given $x,y \in U,$ we write that $x\sim y$ if and only if there is a ball $B\subseteq U$ such that $x, y \in B$.
 \end{definition}
 
 It is easy to check that $\sim$ is an equivalence relation on $U$ (see \cite{LapLu3}), due to the fact that either two balls are disjoint or one is contained in the other. Moreover, there are at most countably many equivalence classes since $\mathbb Q$ is dense in $\mathbb Q_p$. 

  \begin{remark} \label{convex component}(Convex components)
  The equivalence classes of $\sim$ can be thought of as the `convex components' of $U$.  They are an appropriate substitute in the present nonarchimedean context for the notion of connected components, which is not useful in $\mathbb Q_p$ since $\mathbb Z_p$ (and hence, every interval) is totally disconnected.  Note that given any $x\in U,$ the equivalence class (i.e., the \emph{convex component}) of $x$ is the largest  ball containing $x$ (or equivalently, centered at $x$) and contained in $U$. 
    \end{remark}
  
   \begin{definition}\label{p-adic string}
 A \emph{$p$-adic} (or \emph{nonarchimedean}) fractal string
 $ \mathcal{L}_p$ is a bounded open subset $\Omega$ of $\mathbb{Q}_p$. 
 \end{definition}
 Thus it can be written, relative to the above equivalence relation, canonically as a disjoint union of  intervals or balls: 
 \[ \mathcal{L}_p=\bigcup_{j=1}^{\infty} (a_j + p^{n_j} \mathbb{Z}_p)=\bigcup_{j=1}^{\infty} B(a_j, p^{-n_j}). \]
 Here, $B(a_j, p^{-n_j})$ is the largest ball centered at $a_j$ and contained in $\Omega$.
We may assume that the lengths (i.e., Haar measure) of the intervals 
$a_j + p^{n_j} \mathbb{Z}_p$ are nonincreasing, by reindexing if necessary.  That is, 
\begin{equation}\label{sequence of lengths}
p^{-n_1}\geq p^{-n_2} \geq p^{-n_3} \geq \cdots >0.
\end{equation}

\begin{remark}
Ordinary archimedean (or real) fractal strings were introduced in \cite{LapPo, LapPo3} (see also \cite{Lap, Lap3}) and the theory of complex dimensions of those strings was developed in \cite{L-vF2} (and its predecessors).
\end{remark}

\begin{definition}\label{zetaLp}
The\emph{ geometric zeta function} of a $p$-adic fractal string $\mathcal{L}_p$ is defined as
\begin{equation} \label{zeta}
\zeta_{\mathcal{L}_p} (s) = \sum_{j=1}^{\infty} (\mu_H (a_j + p^{n_j} \mathbb{Z}_p))^s
= \sum_{j=1}^{\infty} p^{-n_js} 
\end{equation}
for $\Re(s)$ sufficiently large.
 \end{definition}

\begin{remark}
The geometric zeta function $\zeta_{\mathcal{L}_p}$ is well defined since the decomposition of 
${\mathcal{L}_p}$ into the disjoint intervals $a_j + p^{n_j}\mathbb Z_p$ is unique. Indeed, these intervals are the equivalence classes of which the open set $\Omega$ (defining $\mathcal L_p$) is composed. In other words, they are the $p$-adic  ``convex components'' (rather than the connected components) of $\Omega$. Note that in the real (or archimedean) case, there is no difference between the convex or connected components of $\Omega$, and hence the above construction would lead to the same sequence of lengths as in \cite[\S1.2]{L-vF2}.  
\end{remark}

\begin{figure}[h]\label{screen window}
\psfrag{dots}{$\dots$}
\psfrag{S}{$S$}
\psfrag{W}{$W$}
\psfrag{D}{$D$}
\psfrag{0}{$0$}
\psfrag{1}{$1$}
\raisebox{-1cm}
{\psfig{figure=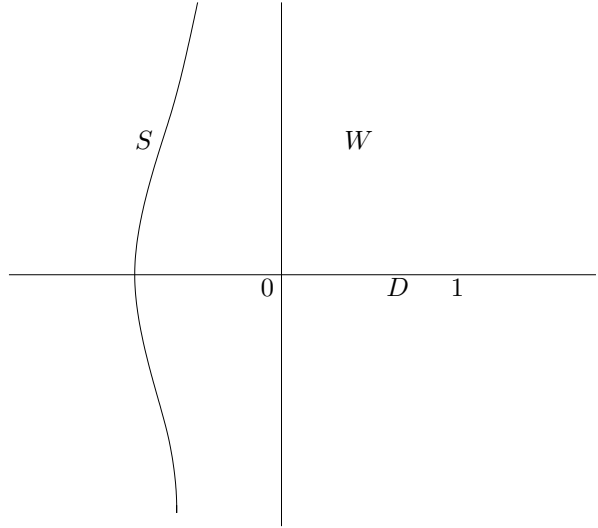, height=7cm}}
\caption{The screen $S$ and the window $W$.}
\end{figure}
  
As in \cite[\S5.3]{L-vF2},  
the \emph{screen} $S$ is the graph (with the vertical and horizontal axes interchanged)
 of a real-valued, bounded and Lipschitz continuous function $S(t)$:
\[
S=\{S(t) + it : t\in \mathbb R\}.
\]
The \emph{window} $W$ is the part of the complex plane to the right of the screen $S$
(see Figure 1): 
\[
W=\{s\in \mathbb C : \Re(s)\geq S(\Im (s))\}.
\]
Let 
\[
\inf S=\inf_{t\in \mathbb R} S(t) \quad \mbox{and} \quad \sup S=\sup_{t\in \mathbb R}S(t), 
\]
and assume that $\sup S \leq \sigma,$ where $\sigma=\sigma_{\mathcal L_p}$ is the abscissa of convergence of $\mathcal L_p$ (to be precisely defined in (\ref{sigma}) below).

  \begin{definition}\label{dvcd} 
  If $\zeta_{\mathcal{L}_p}$ has a meromorphic continuation to an open connected neighborhood of $W\subseteq \mathbb C$, then 
\begin{equation}\label{vcd}
 \mathcal D_{\mathcal L_p}(W)=\{\omega \in W : \omega \mbox{ is a pole of} ~ \zeta_{\mathcal{L}_p}\} 
 \end{equation}
is called the set of \emph{visible complex dimensions} of $\mathcal L_p$.  
If no ambiguity may arise or  if $W=\mathbb C$,  we simply write $\mathcal D_{\mathcal L_p}=\mathcal{D}_{\mathcal{L}_p}(W)$ and call it the set of \emph{complex dimensions} of $\mathcal L_p$. 

Moreover, the \emph{abscissa of convergence} of the Dirichlet series initially defining $\zeta_{\mathcal L_p}$ in Equation (\ref{zeta}) is denoted by
$\sigma=\sigma_{\mathcal L_p}$.
Recall that it is defined by
\begin{equation}\label{sigma}
\sigma_{\mathcal L_p}=\inf\left\{\alpha \in \mathbb R :  \sum_{j=1}^{\infty}p^{-n_j\alpha} <\infty\right\}.
\end{equation}
\end{definition}

\begin{remark}\label{3.7}
In particular, if $\zeta_{\mathcal L_p}$ is entire, which occurs only in the trivial case when 
$\mathcal L_p$ is given by a finite union of intervals, then $\sigma_{\mathcal L_p}=-\infty.$
Otherwise, $\sigma_{\mathcal L_p}\geq 0$ since $\mathcal L_p$ is composed of infinitely many intervals, and hence $\zeta_{\mathcal L_p}(0)=\infty$.  Moreover,  $\sigma_{\mathcal L_p}<\infty$ since $\sigma_{\mathcal L_p}\leq D_M\leq 1,$ where $D_M$ is the Minkowski dimension of $\mathcal L_p$.
(The fact that $D_M\leq 1$ follows since the Haar measure of $\Omega$ is finite and coincides with $\zeta_{\mathcal L_p}(1)$.)
Indeed, as is shown in \cite{LapLu3}, we actually have $\sigma_{\mathcal L_p}=D_M$ for any nontrivial $p$-adic fractal  string.  
 This is the case, for example, for the 3-adic Cantor string introduced in \cite{LapLu1}, for which $\sigma_{\mathcal L_p}=D_M=\log_32$; see Example \ref{cantor string} below.

Observe that since $ \mathcal D_{\mathcal L_p}(W)$ is defined as a subset of the poles of a meromorphic function, it is at most countable. 

Finally, we note that it is well known that $\zeta_{\mathcal L_p}$ is holomorphic for $\Re(s)>\sigma_{\mathcal L_p};$ see, e.g., \cite{Ser}. Hence, 
\[
\mathcal D_{\mathcal L_p} \subset \{s\in \mathbb C: \Re(s)\leq \sigma_{\mathcal L_p}\}.
\]
\end{remark}
  
\begin{example}\label{cantor string}
The 3-adic Cantor string is given by 
\begin{equation}\label{CS_3}
\mathcal{CS}_3=(1+3\mathbb{Z}_3) \cup  (3+9\mathbb{Z}_3) \cup  (5+9\mathbb{Z}_3) \cup \cdots.
\end{equation}
 By definition, the geometric zeta function of $\mathcal{CS}_3$ is given by
  \begin{eqnarray*}
\zeta_{\mathcal{CS}_3}(s)&=&(\mu_H(1+3\mathbb{Z}_3))^s + (\mu_H(3+9\mathbb{Z}))^s + (\mu_H(5+9\mathbb{Z}_3))^s + \cdots\\
&=&\sum_{v=1}^{\infty}\frac{2^{v-1}}{ 3^{vs}} 
=\frac{3^{-s}}{1-2\cdot 3^{-s}} \quad \mbox{for} \quad \Re(s)>\log_32.
\end{eqnarray*}
Hence, by analytic continuation, the meromorphic extension of $\zeta_{\mathcal{CS}_3}$ to the entire complex plane  
$\mathbb C$ exists and  is given by
\begin{equation}\label{cantor zeta}
\zeta_{\mathcal{CS}_3}(s)=\frac{3^{-s}}{1-2\cdot 3^{-s}}, \quad  \mbox{for} ~s\in \mathbb C,
\end{equation}
with poles at 
\[\omega=\frac{\log2}{\log3} +i n \frac{ 2 \pi}{\log3}, \quad n \in \mathbb Z.\] 
 Therefore, the set of complex dimensions  of $\mathcal {CS}_3$ is given by
\begin{equation}\label{cd cantor} 
\mathcal{D}_{\mathcal {CS}_3}=\{   D+i n \mathbf{p}:n \in \mathbb{Z} \},
\end{equation}
where $D=\log_32$ is the dimension of $\mathcal {CS}_3$ and $\mathbf{p}=2\pi / \log3$ is its oscillatory period.  
Moreover, the residue of $\zeta_{\mathcal{CS}_3}(s)$ at $s=D+i n \mathbf{p}$ is given by
\begin{equation}\label{cantor residue}
\res(\zeta_{\mathcal{CS}_3} ; D+i n \mathbf{p})=\frac{1}{2\log3}
\end{equation}
independently of $n\in \mathbb Z$.
Finally, note that $\zeta_{\mathcal{CS}_3}$ is a rational function of $z:=3^{-s}$, i.e.,
\begin{equation*}
\zeta_{\mathcal{CS}_3}(s)=\frac{z}{1-2z}. 
\end{equation*}
\end{example}

The geometric zeta function $\zeta_{\mathcal {CS}_3}$ in Equation  ~\eqref{cantor zeta} is bounded in the left half-plane $\{s\in \mathbb C : \Re (s)\leq0\}$.
In general,
the geometric zeta function of a real or $p$-adic self-similar fractal string is always {\em strongly languid,\/}  i.e., 
\begin{itemize}
\item There exist constants $A,C>0$ such that for all $t\in \mathbb R$ and $m\gg0$, 
\[
 |\zeta_{\mathcal{L}_p}(-m+it)|\leq C A^{|t|}.
\]
\end{itemize}
See \cite[\S5.3]{L-vF2} or \cite{LapLu3} for the general definition of ``languid''.  
  
\section{Volume of  Inner Tubes}\label{inner tube}
 In this section, based on a part of \cite{LapLu3}, we provide a suitable analog in the $p$-adic case of the `boundary' of a fractal string and of the associated inner tubes 
 (inner $\varepsilon$-neighborhoods). Moreover, we give the $p$-adic counterpart of the expression that yields the volume of the inner tubes (see Theorem \ref{thin}). This result serves as a starting point in \cite{LapLu3} for proving the corresponding explicit tube formula. 

\begin{definition}\label{volume definition}
Given a point $a\in \mathbb Q_p$ and a positive real number $r>0$,  
let $B=B(a, r) =\{x\in \mathbb{Q}_p : |x-a|_p \le r \}$
 be a \emph{metrically closed} ball in $\mathbb{Q}_p,$ as above.\footnote
 {Recall that it follows from the ultrametricity of $|\cdot|_p$ that $B$ is topologically both closed and open (i.e., clopen) in $\mathbb Q_p$.}
 We  call 
$S=S(a, r)=\{x\in \mathbb{Q}_p : |x-a|_p = r \}$ the \emph{sphere} of $B$.\footnote
{In our sense, $S$ also coincides with the `metric boundary' of $B$, as given in this definition.}

Let $\mathcal{L}_p= \bigcup_{j=1}^{\infty} B(a_j, r_j)$ be a $p$-adic fractal string. We then define the \emph{metric boundary} $\beta\mathcal{L}_p$ of $\mathcal{L}_p$ to be the disjoint union of the corresponding spheres, i.e.,  
 \[\beta\mathcal{L}_p = \bigcup_{j=1}^{\infty} S(a_j, r_j).\]
Given $\varepsilon>0$,  define the  \emph{thick $p$-adic `inner 
$\varepsilon$-neighborhood'} of $\mathcal{L}_p$ to be
\begin{equation}\label{thick inner tube}
\mathcal{N}_{\varepsilon}=\mathcal{N}_{\varepsilon}(\mathcal L_p):=\{ x\in \mathcal{L}_p : d_p(x, \beta\mathcal{L}_p) < \varepsilon\},
\end{equation}
where $d_p(x, E)=\inf \{ |x-y|_p: y\in E\}$
is the $p$-adic distance of $x \in \mathbb{Q}_p$ to a subset $E \subset \mathbb{Q}_p$.
Then the \emph{volume $\mathcal{V}_{\mathcal L_p}(\varepsilon)$ of the thick inner 
$\varepsilon$-neighborhood} of $\mathcal{L}_p$  is defined to be the Haar measure of $\mathcal{N}_{\varepsilon}$, i.e., 
$\mathcal{V}_{\mathcal L_p}(\varepsilon)=\mu_H(\mathcal{N}_{\varepsilon}).$
\end{definition}

Recall that 
$\zeta_{\mathcal L_p}(1)=\sum_{j=1}^{\infty}p^{-n_j}$ 
is the volume of $\mathcal L_p$ (or rather, of the bounded open subset $\Omega$ of $\mathbb Q_p$ representing $\mathcal L_p$):
$\zeta_{\mathcal L_p}(1)=\mu_H(\mathcal L_p)=\mu_H(\Omega)<\infty.$

\begin{definition}\label{thin inner tube}
Given $\varepsilon >0,$ the \emph{$p$-adic `inner $\varepsilon$-neighborhood'} (or \emph{`inner tube'}) of 
$\mathcal L_p$ is given by 
\begin{equation}\label{equation N}
N_{\varepsilon}=N_{\varepsilon}(\mathcal L_p):=\mathcal N_{\varepsilon}\backslash \beta\mathcal L_p.
\end{equation}
Then the \emph{volume $V_{\mathcal L_p}(\varepsilon)$ of the inner 
$\varepsilon$-neighborhood} of $\mathcal{L}_p$  is defined to be the Haar measure of $N_{\varepsilon}$, i.e., 
\begin{equation}\label{volume}
V_{\mathcal L_p}(\varepsilon):=\mu_H(N_{\varepsilon})=\mathcal V_{\mathcal L_p}(\varepsilon)-\mu_H(\beta \mathcal L_p).
\end{equation}
\end{definition}

We next state the nonarchimedean counterpart  of \cite[Eq.~(3.2)]{LapPo}
 (see also \cite[Eq. (8.1)]{L-vF2}), which is the key result in \cite{LapLu3} that will enable us to obtain an appropriate $p$-adic analog of the fractal tube formula as well as of the notion of Minkowski dimension and content (see \S\ref{exact tf} and \S\ref{amc}).

\begin{theorem}[Volume of  inner tubes]\label{thin}
Let $ \mathcal{L}_p=\bigcup_{j=1}^{\infty}  B(a_j, p^{-n_j})$ be a $p$-adic fractal string. Then, for any $\varepsilon >0,$ we have
\begin{eqnarray}\label{volumeequation}
V_{\mathcal L_p}(\varepsilon)
&=& p^{-1}\left(\zeta_{\mathcal L_p}(1)-\sum_{j=1}^k   p^{-n_j}\right),\label{equation2}
    \end{eqnarray} 
  where $k=k(\varepsilon)$ is the largest   integer such that
 $n_k \leq \log_p\varepsilon^{-1}$.
\end{theorem}

\begin{remark}\label{limit}
Note that  $\lim_{\varepsilon \to 0^+}V_{\mathcal L_p}(\varepsilon)=0,$ which justifies  Definition \ref{thin inner tube}; see \cite{LapLu3}. Further observe that even though `the' metric boundary may depend on the choice of the centers $a_j$ ($j\in \mathbb N^*$), both $\mathcal V_{\mathcal L_p}(\varepsilon)$ and  $V_{\mathcal L_p}(\varepsilon)$ are indepedent of this choice (in light of Equation (\ref{volumeequation})). 
\end{remark}

\begin{example}[The explicit tube formula for 3-adic Cantor string]\label{cantor volume}
Let $\varepsilon >0$.  Then, by Theorem \ref{thin}, we have 
\begin{equation}
V_{\mathcal {CS}_3}(\varepsilon)=\frac{1}{3}\sum_{n=k+1}^{\infty}
\frac{2^{n-1}}{3^n}
=\frac{1}{3}\left( \frac{2}{3}\right)^k,
\end{equation}
where $k:=[\log_3\varepsilon^{-1}]$. 
Let $x:=\log_3\varepsilon^{-1}=k+\{x\}$, where $\{x\}$ is the fractional part of $x$. Then a simple computation shows that
$\left(\frac{2}{3}\right)^x=\varepsilon^{1-D} ~\mbox{and}~ e^{2\pi i n x}=\varepsilon^{-in\mathbf{p}}$, 
with $D=\log_32$ and $\mathbf{p}=2\pi /\log 3$ as in Example \ref{cantor string}.
Using the Fourier expansion for the periodic function $b^{-\{x\}}$, as given by \cite[Eq. (1.13)]{L-vF2}, for $b=3^{-1}$
and the above value of $x$, we obtain an expansion in terms of the complex dimensions $\omega=D+in\mathbf{p}$ of $\mathcal{CS}_3$:

\begin{eqnarray}
V_{\mathcal{CS}_3}(\varepsilon)\nonumber
&=& \frac{3^{-1}}{2\log 3}\sum_{n\in \mathbb Z}\frac{\varepsilon^{1-D-in\textbf{p}}} {1-D-in\textbf{p}}\\
&=& \frac{1}{ 6\log 3} \sum_{\omega \in \mathcal D_{\mathcal{CS}_3}}\frac{\varepsilon^{1-\omega  }}{1-\omega}
\label{tfc2}
\end{eqnarray}
since $\mathcal D_{\mathcal{CS}_3}$ is given by (\ref{cd cantor}).
\end{example}
  
 \section{Explicit Tube Formulas for $p$-Adic Fractal Strings} \label{explicit tf} 
 The following result is the counterpart in this context of Theorem 8.1 of \cite{L-vF2}, the distributional tube formula for real fractal strings.  It is established in \cite{LapLu3} by using, in particular, the extended distributional explicit formula of \cite[Thms.~5.26 and 5.27]{L-vF2}, along with the expression for the volume of thin inner $\varepsilon$-tubes stated in Theorem \ref{thin}.

\begin{theorem}%[$p$-Adic explicit  tube formula]
\label{dtf}
Let $\mathcal L_p$ be a languid $p$-adic fractal string.
Further assume that $\sigma_{\mathcal L_p}<1.$\footnote
{Recall from Remark \ref{3.7} that we always have $\sigma_{\mathcal L_p}\leq 1$. Moreover, if $\mathcal L_p$ is self-similar,  then $\sigma_{\mathcal L_p}< 1$ (in light of \cite{PW} and the definition of $\sigma_{\mathcal L_p}$).}
Then the volume of the thin inner  $\varepsilon$-neighborhood of $\mathcal L_p$ is given by
\begin{equation}\label{detf}
V_{\mathcal L_p}(\varepsilon)= \sum_{ \omega \in \mathcal{D}_{\mathcal{L}_p} (W)} \res 
\left (
\frac{p^{-1} \zeta_{\mathcal{L}_p} (s) \varepsilon^{1-s}} {1-s}; \omega
\right) 
+ \mathcal{R}_p (\varepsilon),
\end{equation}
where $\mathcal D_{\mathcal L_p}(W)$ is the set of visible complex dimensions of $\mathcal L_p.$ %(as given in Definition \ref{dvcd}).
Here, the distributional error term is given by 
\begin{equation}\label{error term}
\mathcal R_p(\varepsilon)=\frac{1}{2\pi i}
\int_{S}\frac{p^{-1} \zeta_{\mathcal{L}_p} (s) \varepsilon^{1-s}} {1-s}ds
\end{equation}
and is estimated distributionally\footnote
{As in \cite[Defn.~5.29]{L-vF2}.}
 by
\begin{equation}\label{estimate}
\mathcal{R}_p(\varepsilon)=O(\varepsilon^{1-\sup S}), \hspace{1cm} \mbox{as} ~\varepsilon \rightarrow 0^+.
\end{equation}
Moreover, if $\mathcal L_p$ is strongly languid {\rm(}which is the case of all $p$-adic self-similar strings; see \S3 and \S9{\rm)}, then we can take $W=\mathbb C$ and
$\mathcal{R}_p(\varepsilon)\equiv 0$.
\end{theorem}

 \begin{corollary}%[$p$-Adic fractal tube formula]
 \label{ftf}
 If, in addition to the hypotheses in Theorem \ref{dtf}, we assume that all the visible complex dimensions of 
 $\mathcal{L}_p$ are  simple, then  
\begin{equation}\label{pftf} 
V_{\mathcal L_p}(\varepsilon)= \sum_{ \omega \in \mathcal{D}_{\mathcal L_p}(W)} 
c_{\omega} \frac{ \varepsilon^{1-\omega}} {1-\omega } 
+ \mathcal{R}_p (\varepsilon),
\end{equation}
where $c_{\omega}=p^{-1}\res\left(\zeta_{\mathcal L_p}; \omega\right)$. 
Here, the error term $\mathcal R_p$ is given by {\rm(}\ref{error term}{\rm)}
and is estimated by {\rm(}\ref{estimate}{\rm)} in the languid case. 
Furthermore, we have $ \mathcal{R}_p (\varepsilon) \equiv 0$ in the strongly languid case 
{\rm(}yielding an \emph{exact} tube formula{\rm)},  provided we choose $W=\mathbb C$. 
 \end{corollary}

 \begin{remark}
 In \cite [Ch.~8]{L-vF2}, under different sets of assumptions, both distributional and pointwise tube formulas are obtained for archimedean fractal strings (and also, for archimedean self-similar fractal strings). 
 (See, in particular, Theorems 8.1 and 8.7, along with \S8.4 in \cite{L-vF2}.)
 At least for now, in the nonarchimedean case, we limit ourselves to discussing distributional explicit tube formulas. We expect, however, that under appropriate hypotheses, one should be able to obtain a pointwise fractal tube formula for $p$-adic fractal strings and especially, for $p$-adic self-similar strings. 
 In fact, for the simple examples of the nonarchimedean Cantor and Fibonacci strings, the direct derivation of the fractal tube formula (\ref{pftf})
 yields a formula that is valid pointwise and not just distributionally. (See, in particular, Examples  
  \ref{cantor volume} and \ref{fibonacci}.) We leave the consideration of such possible extensions to a future work.
 \end{remark}

 \begin{example}[The explicit tube formula for 3-adic Cantor string revisited]
 By Equation (\ref{cantor residue}), we have that
 \[\res(\zeta_{\mathcal{CS}_3}; \omega)=\frac{1}{2\log3},\]
 independently of $\omega \in \mathcal{D}_{\mathcal{CS}_3 }$. So, using the last part of Theorem \ref{dtf},  the exact fractal tube formula for the 3-adic Cantor string is found to be 
\begin{equation}\label{tfc}
V_{\mathcal{CS}_3}(\varepsilon)=
\frac{3^{-1}}{2\log3} \sum_{\omega \in \mathcal D_{\mathcal {CS}_3}}\frac{\varepsilon^{1-\omega  }}{1-\omega},
\end{equation}
which is exactly the same as Equation (\ref{tfc2}). 
 
Note that since $\mathcal{CS}_3$ has simple complex dimensions, we may also apply Corollary \ref{ftf} (in the strongly languid case when $W=\mathbb C$) in order  to precisely recover Equation (\ref{tfc}). (Alternatively, we could use Corollary \ref{5.13} in \S\ref{exact tf} below.)

We may rewrite (\ref{tfc2}) or (\ref{tfc}) in the following form (which agrees with the tube formula to be obtained in Corollary \ref{5.13}):
\[
V_{\mathcal{CS}_3}(\varepsilon)=\varepsilon^{1-D}G_{\mathcal{CS}_3}(\log_{3}\varepsilon^{-1}),
\]
where  $G_{\mathcal{CS}_3}$ is the nonconstant periodic function of period 1 on $\mathbb R$ given by 
\[
G_{\mathcal{CS}_3}(x):=\frac{1}{6\log3}\sum_{n\in \mathbb Z}\frac{e^{2\pi inx}}{1-D-in\mathbf{p}}.
\]
Finally, we note that since the Fourier series 
\[
\sum_{n\in \mathbb Z}\frac{e^{2\pi inx}}{1-D-in\mathbf{p}}
\]
is pointwise convergent on $\mathbb R$, the above direct computation of $V_{\mathcal{CS}_3}(\varepsilon)$
shows that (\ref{tfc2}) and (\ref{tfc}) actually hold pointwise rather than distributionally.
\end{example}

 \section{Nonarchimedean Self-Similar  Strings }\label{sss}
 Nonarchimedean (or $p$-adic) self-similar strings form an important class of $p$-adic fractal strings. In this section, we first recall the construction of these strings, as provided in \cite{LapLu2} and \cite{LapLu4}; see  \S\ref{construction}. Later on, we will give an explicit expression for their geometric zeta functions and deduce from it the periodic structure of their poles (or complex dimensions) and zeros, as obtained in 
 \cite{LapLu2}; see \S \ref{gzf sss}--\ref{zeros and poles}. 
 Moreover, in \S\ref{exact tf}, we will deduce from the results of \S\ref{explicit tf} and 
 \S\ref{gzf sss}--\ref{PeriodicityR} the special form of the fractal tube formula for $p$-adic self-similar strings. Finally, in \S\ref{amc}, we will apply this latter result in order to calculate the average Minkowski content of such strings.
 
 \subsection{Geometric Construction}\label{construction}
 
 Before explaining how to construct arbitrary $p$-adic self-similar strings, we need to introduce a definition and a few facts pertaining to $p$-adic similarity transformations.\footnote{The standard definition of self-similarity (in Euclidean space or in more general complete metric spaces) can be found in \cite{Hut} and in \cite{Fal}, for example.} 
 
 \begin{definition}\label{similarity mapping}
A map $\Phi: \mathbb Z_p\longrightarrow \mathbb Z_p$ is called a \emph{similarity contraction mapping} of $\mathbb{Z}_p$ if there is a real number $r \in (0,1)$ such that 
\[
|\Phi(x)-\Phi(y)|_p=r\cdot|x-y|_p,
\]
for all $x, y \in \mathbb Z_p$. 
\end{definition}
 
Unlike in Euclidean space (and in the real line $\mathbb R$, in particular), it is not true that   every similarity transformation of $\mathbb Q_p$ (or of $\mathbb Z_p$) is necessarily affine.
Actually, in the nonarchimedean world (for example, in $\mathbb Q_p^d$, with $d\geq 1$), and in the $p$-adic line $\mathbb Q_p$, in particular, there are a lot of similarities which are not affine. 
However, it is known (see, e.g., \cite{Sch}) that every analytic similarity must be affine.\footnote
{Here, a map $f:\mathbb Q_p \longrightarrow \mathbb Q_p$ is said to be analytic if it admits a convergent power series expansion about 0, and with coefficients in $\mathbb Q_p$, that is convergent in all of $\mathbb Q_p$.}
    Hence, from now on, we will be working with a similarity contraction mapping 
     $\Phi: \mathbb Z_p\longrightarrow \mathbb Z_p$ that is affine. 
     Thus we assume that there exist constants $a,b\in \mathbb Z_p$ with $|a|_p<1$ such that    
    $\Phi(x)=ax+b$ for all $x\in \mathbb Z_p$. Regarding the scaling factor $a$ of the contraction, it is well known that it can be written as $a=u\cdot p^n,$ for some unit $u\in \mathbb{Z}_p$ (i.e., $|u|_p=1$) and $n\in \mathbb N^{*}$ (see \cite{Neu}).  Then $r=|a|_p=p^{-n}.$  We summarize this fact in the following lemma:
 
  \begin{lemma}\label{scaling} 
  Let $\Phi(x)=ax+b$ be an affine similarity contraction mapping of $\mathbb {Z}_p$ with the scaling ratio $r$. Then $b\in \mathbb Z_p$ and  $a\in p\mathbb Z_p,$ and the scaling factor is $r=|a|_p=p^{-n}$ for some $n\in \mathbb N^{*}$. 
     \end{lemma}
 
 \begin{figure}[ht]\label{p-adic self-similar construction}
\psfrag{dots}{$\vdots$}
\psfrag{Zp}{$\mathbb Z_p$}
\psfrag{P1}{$\Phi_1(\mathbb Z_p$)}
\psfrag{dots}{$\cdots$}
\psfrag{PJ}{$\Phi_N(\mathbb Z_p$)}
\psfrag{G1}{$G_1$}
\psfrag{GK}{$G_K$}
\psfrag{P11}{$\Phi_{11}\mathbb Z_p$}
\psfrag{P1J}{$\Phi_{1N}\mathbb Z_p$}
\psfrag{G11}{$\Phi_1G_1$}
\psfrag{G1K}{$\Phi_1G_K$}
\psfrag{PJ1}{$\Phi_{N1}\mathbb Z_p$}
\psfrag{PJJ}{$\Phi_{NN} \mathbb Z_p $}
\psfrag{GJ1}{ $\Phi_NG_1$}
\psfrag{GJK}{ $\Phi_NG_K$}
\psfrag{vdots}{$\vdots$}
\raisebox{-1cm}
{\psfig{figure=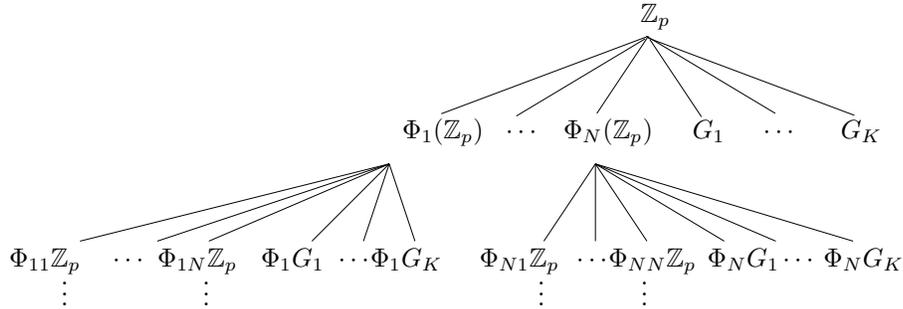, height=4cm}}
\caption{Construction of a  $p$-adic self-similar fractal string.} \end{figure}

For simplicity, let us take the unit interval (or ball) $\mathbb{Z}_p$ in $\mathbb{Q}_p$ and construct a 
\emph{ $p$-adic} (or\emph{ nonarchimedean}) \emph{self-similar  string} $\mathcal{L}_p$ as follows (see \cite{LapLu2}).\footnote
{In the sequel, $\mathcal L_p$ is interchangeably called a \emph{$p$-adic} or \emph{nonarchimedean} self-similar string.}
  Let $N\geq2$ be an integer and  $\Phi_1, \dots, \Phi_N: \mathbb Z_p\longrightarrow \mathbb Z_p$ be $N$ affine similarity contraction mappings with the respective scaling ratios $r_1, \dots, r_N \in (0,1)$ satisfying 
\begin{equation}\label{ratios}
1>r_1\geq r_2\geq \cdots \geq r_N>0;
\end{equation}
see Figure 2.  Assume that 
\begin{equation}\label{rho}
\sum_{j=1}^N r_j <1,
\end{equation}
and the images $\Phi_j(\mathbb{Z}_p)$ of $\mathbb{Z}_p$ do not overlap, i.e., 
$\Phi_j(\mathbb{Z}_p) \cap \Phi_l(\mathbb{Z}_p)= \emptyset$
for all $j\neq l$. 
Note that it follows from Equation (\ref{rho}) that $\bigcup_{j=1}^N \Phi_j(\mathbb Z_p)$ is not all of $\mathbb Z_p$.
We therefore have the following (nontrivial) decomposition of
$\mathbb Z_p$ into disjoint $p$-adic intervals:
\begin{equation}\label{first generation}         
 \mathbb Z_p = \bigcup_{j=1}^N \Phi_j(\mathbb Z_p)\cup \bigcup_{k=1}^K G_k,
\end{equation}
where $G_k$ is defined below. 

In a procedure reminiscent of the construction of the ternary Cantor set, we then subdivide the interval $\mathbb{Z}_p$ by means of the subintervals $\Phi_j(\mathbb{Z}_p)$.
Then the convex\footnote
{We choose the convex components instead of the connected components because $\mathbb{Z}_p$ is totally disconnected. Naturally, no such distinction is necessary in the archimedean case;  cf.~\cite[\S2.1.1]{L-vF2}.
 Here and elsewhere in this paper,
  a subset $E$ of $\mathbb Q_p$ is said to be `convex' if for every $x,y\in E$,
   the $p$-adic segment $\{tx+(1-t)y: t\in \mathbb Z_p\}$ lies entirely in $E$.} 
components of 
\[
\mathbb Z_p \backslash \bigcup_{j=1}^N \Phi_j(\mathbb Z_p)
\]
are the first \emph{substrings} of the $p$-adic  self-similal string $\mathcal L_p$, say $G_1, G_2, \ldots, G_K,$ with $K\geq 1$.  These intervals $G_k$ are called the \emph{generators},  the deleted intervals in the first generation of the construction of $\mathcal L_p.$\footnote
{Their archimedean counterparts are called `gaps' in  \cite[Ch.~2 and \S8.4]{L-vF2},
 where archimedean self-similar strings are introduced.}
  The length of each $G_k$ is denoted by $g_k$; so that $g_k=\mu_H(G_k)$.\footnote
{We note that the lengths $g_k$ ($k=1,2,\ldots, K$) will sometimes be called the (nonarchimedean) `gaps' or `gap sizes' in the sequel.}
Without loss of generality, we may assume that the lengths
$g_1, g_2, \dots, g_K$ of the first substrings (i.e., intervals) of $\mathcal L_p$ satisfy
\begin{equation}\label{gaps}
1>g_1\geq g_2\geq \cdots \geq g_K >0. 
\end{equation}
It follows from Equation (\ref{first generation}) and the additivity of Haar measure $\mu_H$ that 
\begin{equation}\label{gap identity}
\sum_{j=1}^N r_j + \sum_{k=1}^K g_k =1.
\end{equation}
We then repeat this process with each of the remaining subintervals 
$\Phi_j(\mathbb{Z}_p)$ of $\mathbb Z_p,$ for $j=1,2,\ldots,N$. 
And so on, ad infinitum. 
As a result, we obtain a \emph{$p$-adic  self-similar  string} $\mathcal{L}_p=l_1, l_2, l_3, \dots,$
consisting of intervals of length $l_n$ given by 
\begin{equation}\label{length}
r_{\nu_1}r_{\nu_2}\cdots r_{\nu_q}g_k, 
\end{equation}
for $k=1, \ldots, K$ and all choices of $q\in \mathbb N$ and $\nu_1, \ldots, \nu_q \in \{1, \ldots, N\}$.
Thus, the lengths are of the form 
$r_1^{e_1}\dots r_N^{e_N}g_k$ with $e_1, \dots, e_N \in \mathbb{N}$ (but not all zero).

In \cite{LapLu2}, the classic notion of self-similarity is extended to the nonarchimedean setting, much as in \cite{Hut}, where the underlying complete metric space is allowed to be arbitrary. We note that the next result follows by applying the classic Contraction Mapping Principle to the complete metric space of all nonempty compact subsets of $\mathbb Z_p.$ (Note that  $\mathbb{Z}_p$ itself is complete since it is a compact metric space.)

 \begin{theorem}%[$p$-Adic self-similar set]
  \label{Hutchinson}
There is a unique nonempty compact subset $\mathcal S_p$ of $\mathbb Z_p$ such that 
\[
\mathcal S_p=\bigcup_{j=1}^N \Phi_j(\mathcal S_p).
\]
The set $\mathcal S_p$ is called the $p$-adic self-similar set associated with the self-similar system 
$\mathbf\Phi=\{\Phi_1, \ldots, \Phi_N\}.$ It is also called the $\mathbf\Phi$-invariant set.
\end{theorem}

The relationship between the $p$-adic self-similar string $\mathcal L_p$ and the above $p$-adic self-similar set $\mathcal S_p$ is given by the following theorem, also obtained in \cite{LapLu2}:\footnote
{In Theorem \ref{complementary},
 $\mathcal L_p$ is not viewed as a sequence of lengths but is viewed instead as the open set which is canonically given by a disjoint union of intervals (its $p$-adic convex components), as described in the above construction of a $p$-adic self-similar string.}

\begin{theorem}\label{complementary}
 {\rm(}i{\rm)}  $\mathcal L_p=\mathbb Z_p \backslash \mathcal S_p,$
the complement of $\mathcal S_p$ in $\mathbb Z_p$. 

 {\rm(}ii{\rm)} 
 $\mathcal{L}_p=\bigcup_{\alpha=0}^{\infty}\bigcup_{w\in W_{\alpha}}\bigcup_{k=1}^K\Phi_w(G_k),$
while
$\mathcal S_p=\bigcap_{\alpha=0}^{\infty}\bigcup_{w\in W_{\alpha}} \Phi_w(\mathbb{Z}_p),$
where $W_{\alpha}=\{1,2, \ldots, N\}^{\alpha}$ denotes the set of all finite words on $N$ symbols, of length $\alpha,$ and
 $\Phi_w:=\Phi_{w_{\alpha}} \circ \cdots  \circ \Phi_{w_1}$ 
 for~ $w=(w_1, \dots, w_{\alpha})\in W_{\alpha}.$
\end{theorem}

\begin{figure}[h]
\psfrag{dots}{$\vdots$}
\psfrag{Z3}{$\mathbb Z_3$}
\psfrag{Z0}{$0+3\mathbb Z_3$}
\psfrag{Z1}{$1+3\mathbb Z_3=G$}
\psfrag{Z2}{$2+3\mathbb Z_3$}
\psfrag{1}{$0+9\mathbb Z_3$}
\psfrag{2}{$\Phi_1(G)$}
\psfrag{3}{$6+9\mathbb Z_3$}
\psfrag{4}{$2+9\mathbb Z_3$}
\psfrag{5}{$\Phi_2(G)$}
\psfrag{6}{$8+9\mathbb Z_3$}
\raisebox{-1cm}
{\psfig{figure=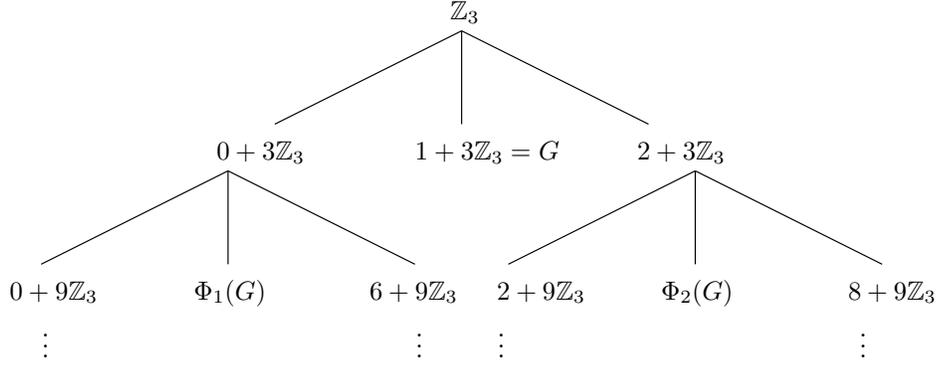, height=4.9cm}}
\caption{Construction of the 3-adic Cantor string $\mathcal{CS}_3$ via an IFS.}
\end{figure}

\begin{example}[Nonarchimedean Cantor string as a 3-adic self-similar string]\label{cantor}
In this example, we review the construction of the nonarchimedean Cantor string 
$\mathcal{CS}_3$, as introduced in \cite{LapLu1} and revisited in \cite{LapLu2}. Our main point here is to stress the fact that $\mathcal{CS}_3$ is a special case of a $p$-adic self-similar string, as constructed just above, and to prepare the reader for more general results about nonarchimedean self-similar strings, as obtained in the rest of this paper. 

Let $\Phi_1, \Phi_2: \mathbb Z_3\longrightarrow \mathbb Z_3$ be the two affine similarity contraction mappings of $\mathbb {Z}_3$ given by 
\begin{equation}\label{maps}
\Phi_1(x)=3x \quad \mbox{and}\quad \Phi_2(x)=2 + 3x,
\end{equation}
with the same scaling ratio $r=3^{-1}$ (i.e., $r_1=r_2=3^{-1}$).
By analogy with the construction of the real Cantor string, subdivide the interval $\mathbb{Z}_3$ into subintervals 
\[
\Phi_1(\mathbb{Z}_3)=0+3\mathbb Z_3 \quad \mbox{and} \quad
 \Phi_2(\mathbb{Z}_3)=2+3\mathbb Z_3.\]
The remaining (3-adic) convex component
\[
\mathbb Z_3 \backslash \bigcup_{j=1}^2 \Phi_j(\mathbb Z_3)=1+3\mathbb Z_3=G
\]
is the first substring of a  $3$-adic self-similar string, called the \emph{nonarchimedean} 
 \emph{Cantor string} and denoted by $\mathcal{CS}_3$ \cite{LapLu1}. 
The length of $G$ is
 $l_1=\mu_H(1+3\mathbb{Z}_p)=3^{-1}.$
By repeating this process with the remaining subintervals 
$\Phi_j(\mathbb{Z}_3),~ \mbox{for} ~ j=1,2$, and continuing on, ad infinitum, we eventually obtain a sequence 
$\mathcal{CS}_3 =l_1, l_2, l_3, \dots,$
associated with the open set resulting from this construction and consisting of intervals of lengths $l_v=3^{-v}$ with multiplicities $m_v=2^{v-1}$, for $v\in \mathbb N^*$. 
As follows from this construction (see Figure 3 and Equation (\ref{maps}), along with part (ii) of Theorem \ref{complementary}), the nonarchimedean Cantor string  $\mathcal{CS}_3$ can also be written as 
\begin{equation}\label{CS_3}
\mathcal{CS}_3=(1+3\mathbb{Z}_3) \cup  (3+9\mathbb{Z}_3) \cup  (5+9\mathbb{Z}_3) \cup \cdots.
\end{equation}

We refer the interested reader to \cite{LapLu1} and \cite{LapLu4} for additional information concerning the nonarchimedean Cantor string $\mathcal{CS}_3$ and the associated nonarchimedean Cantor set $\mathcal C_3$. 
We just mention here that in light of part (i) of Theorem \ref{complementary}, we can recover the 3-adic Cantor set $\mathcal C_3$ as the complement of the 3-adic Cantor string $\mathcal {CS}_3$ in the unit interval (and vice-versa):
\begin{equation}
\mathcal{CS}_3=\mathbb Z_3 \backslash \mathcal C_3, \quad \mbox{and so} \quad 
\mathcal C_3=\mathbb Z_3 \backslash \mathcal{CS}_3.
\end{equation}
Indeed, according to Theorem \ref{Hutchinson}, $\mathcal C_3$ is the self-similar set associated with the self-similar system  $\mathbf{\Phi}=\{\Phi_1, \Phi_2\}$.
\end{example}
 
\section{Geometric Zeta Function of  $p$-Adic Self-Similar  Strings}\label{gzf sss}
  In this section, as well as in \S\ref{PeriodicityR} and \S\ref{zeros and poles}, we will survey results obtained in \cite{LapLu2} about the geometric zeta functions and the complex dimensions of $p$-adic self-similar strings. (See also \cite{LapLu4}, where the archimedean and nonarchimedean situations are contrasted.)
  
  In the next theorem, we provide a first expression for the geometric zeta function of a nonarchimedean self-similar string. At first sight, this expression is almost identical to the one obtained in the archimedean case in \cite[Thm.~2.4]{L-vF2}. 
  Later on, however, we will see that unlike in the archimedean case where the situation is considerably more subtle and complicated (cf.~\cite[Thms.~2.17 and 3.6]{L-vF2}), this expression can be significantly  simplified since the two potentially transcendental functions appearing in the denominator and numerator of Equation (\ref{sszf}) below can always be made rational; see Theorem 
  \ref{rationality} in \S\ref{PeriodicityR}.

\begin{theorem}\label{geometric zeta function}
Let $\mathcal{L}_p$ be a  $p$-adic self-similar  string with scaling ratios $\{r_j\}_{j=1}^N$ and gaps $\{g_k\}_{k=1}^K$, as in the above construction. 
Then the geometric zeta function of $\mathcal{L}_p$
has a meromorphic extension to the whole complex plane $\mathbb{C}$ and is given by
\begin{equation}\label{sszf}
\zeta_{\mathcal{L}_p}(s)=
\frac
{\sum_{k =1}^{K} g_{k}^s}
{1-\sum_{j=1}^N r_j ^s}, 
\quad  \mbox{for} \quad  s\in \mathbb{C}.
\end{equation}
\end{theorem}
\begin{corollary}
The set of complex dimensions of a $p$-adic  self-similar fractal  string $\mathcal{L}_p$ is contained in the set of complex solutions $\omega$ of the Moran equation $\sum_{j=1}^N r_j ^{\omega}=1.$ If the string has a single generator {\rm(}i.e., if $K=1${\rm)}, then this inclusion is an equality.\footnote
{See, e.g.,  Examples \ref{cantor}, \ref{fibonacci} and Theorem \ref{rationality}.}
\end{corollary} 

\begin{definition}\label{latticenonlattice} 
A $p$-adic  self-similar  string $\mathcal{L}_p$ is said to be \emph{lattice}  if the multiplicative group generated by the  scaling ratios 
$r_1, r_2, \ldots, r_N$ is discrete  in $(0,\infty)$. 
Otherwise, $\mathcal L_p$ is said to be \emph{nonlattice}. Furthermore, $\mathcal L_p$ is said to be \emph{strongly lattice} if  the multiplicative  group generated by $\{r_1, \ldots, r_N, g_1, \dots, g_K\}$ is discrete in  $(0,\infty)$. Naturally, a strongly lattice string is also a lattice string. 
\end{definition}

 \begin{theorem}\label{lattice}
Every  $p$-adic self-similar fractal string is strongly lattice.  
\end{theorem}

 \begin{remark}\label{remark 4.11} 
 Theorem \ref{lattice} follows from the fact that all the scaling ratios $r_j$ and  the gaps $g_k$ must  belong to  the group $p^{ \mathbb Z },$ as will be discussed below in more detail in \S\ref{PeriodicityR}. It follows that  $p$-adic self-similar  strings are lattice strings in a very strong sense, namely, their geometric zeta functions are \emph{rational functions} of a suitable variable $z$ (see Theorem \ref{rationality} below).  \end{remark}

\begin{remark}Theorem \ref{lattice} is in sharp contrast with the usual theory of real 
self-similar strings developed in \cite[Chs.~2 and 3]{L-vF2}.  Indeed,  there are both lattice and nonlattice strings  in the archimedean case. Furthermore, generically, archimedean self-similar strings are nonlattice. Moreover, it is shown in \cite[Ch.~3]{L-vF2} by using Diophantine approximation that every nonlattice string in $\mathbb R=\mathbb Q_{\infty}$ can be approximated by a sequence of lattice strings with oscillatory periods increasing to infinity. It follows that the complex dimensions of an archimedean nonlattice string are quasiperiodically distributed (in a very precise sense, that is explained in \emph{loc.~cit.}) because the complex dimensions of archimedean lattice strings are periodically distributed along finitely many vertical lines. Clearly, there is nothing of this kind in the nonarchimedean case since $p$-adic self-similar strings are necessarily lattice. 
\end{remark}

\section{Rationality of the Geometric Zeta Function}\label{PeriodicityR}
  In this section, we show that the geometric zeta function of a $p$-adic self-similar string is always rational (after an appropriate change of variable). 
  It will follow (see Theorem \ref{periodicity}) that \emph{not only the poles} (i.e., the complex dimensions of $\mathcal{L}_p$) \emph{but also the zeros of $\zeta_{\mathcal{L}_p}$} \emph{are periodically distributed.} 

We introduce some necessary notation. First, by Lemma \ref{scaling}, we can write
\begin{equation*}\label{r}
r_j=p^{-n_j}, \quad \mbox{with} \quad n_j\in \mathbb N^* \quad  \mbox{for} \quad j=1, 2, \ldots, N.
\end{equation*}
Second, we write 
\begin{equation*}
g_k=\mu_H(G_k)=p^{-m_k}, \quad \mbox{with} \quad m_k\in \mathbb N^* \quad \mbox{for} \quad k=1, 2, \ldots, K.
\end{equation*}
Third, let 
\begin{equation*}
d= \mbox{gcd}\{n_1, \ldots, n_N, m_1, \ldots, m_K\}.
\end{equation*}
Then there exist positive integers $n_j' ~\mbox{and}~ m_k'$ such that
\begin{equation}\label{n'}
n_j=dn_j'  \quad \mbox{and} \quad m_k=dm_k' \quad \mbox{for}  \quad j=1, \ldots, N  \quad \mbox{and} \quad  k=1,\ldots, K.
\end{equation}
Finally, we set\footnote
{Note that by construction,
 $r_j=r^{n_j'}$ and $g_k=r^{m_k'}$ for $j=1, \dots, N$ and $k=1,\dots, K.$ Hence, $r=p^{-d}$ is the multiplicative generator in $(0,1)$ of the rank one group generated by $\{r_1, \ldots, r_N, g_1, \dots, g_K\}$ (or,
  equivalently, by either
 $\{r_1, \ldots, r_N\}$ or $\{g_1, \ldots, g_K\}$).}
 
\begin{equation}\label{Q}
p^d=1/r.
\end{equation}
Without loss of generality, we may assume that the scaling ratios $r_j$ and  the gaps $g_k$ are written in nonincreasing order as in Equations (\ref{ratios}) and (\ref{gaps}), respectively; so that 
\begin{equation}
0<n_1'\leq n_2' \leq \cdots \leq n_N'  \quad \mbox{and} \quad 0<m_1'\leq m_2' \leq \cdots \leq m_K'. 
\end{equation}

\begin{theorem}\label{rationality}
Let $\mathcal{L}_p$ be a $p$-adic self-similar  string and $z=r^s,$ with $r=p^{-d}$ as in Equation 
{\rm(}\ref{Q}{\rm)}.
Then the geometric zeta function
 $\zeta_{\mathcal{L}_p}$ of $\mathcal{L}_p$ is a rational function in $z$. Specifically, 
 \begin{equation}\label{rational function}
\zeta_{\mathcal{L}_p}(s)=
\frac
{\sum_{k =1}^{K} z^{m_{k}'}}
{1-\sum_{j=1}^N z ^{n_j'}},
\end{equation}
where $m_k', n_j'\in \mathbb N^*$ are given by Equation {\rm(}\ref{n'}{\rm)}.
\end{theorem}

 \begin{definition}\label{oscillatoryperiod} 
 Let $\mathbf{p}=\frac{2\pi}{d\log p}$. Then $\mathbf{p}$ is called the \emph{oscillatory period} of $\mathcal{L}_p.$ \end{definition}

\subsection{Periodicity of the Poles and the Zeros of $\zeta_{\mathcal L_p}$}\label{zeros and poles}

The following result (also from \cite{LapLu2}) is the nonarchimedean counterpart of \cite[Thms.~2.17 and 3.6]{L-vF2}, which provide the rather subtle structure of the complex dimensions of archimedean self-similar strings. It is significantly simpler, however, due to the fact that nonlattice $p$-adic self-similar strings do not exist.

To avoid any confusion, we stress that in the statement of the next theorem, $\zeta_{\mathcal L_p}$ is viewed as a function of the original complex variable $s$. Moreover,  as was recalled in Remark \ref{3.7}, 
it follows from a theorem in \cite{LapLu3}  that 
 the dimension of $\mathcal L_p$ defined  as the \emph{Minkowski dimension} $D=D_{\mathcal L_p}$ coincides with  the \emph{abscissa of convergence} of the Dirichlet series originally defining $\zeta_{\mathcal L_p}$ and denoted (as in Equation (\ref{sigma})) by $\sigma=\sigma_{\mathcal L_p} $. Furthermore,  
  let $\delta$ be the \emph{similarity dimension} of $\mathcal L_p$, i.e., the unique \emph{real} (and hence, positive) solution of the Moran equation $\sum_{j=1}^Nr_j^s=1$; then $\delta=D$ by part (iii) of Theorem \ref{periodicity} below. Therefore, in the present case of $p$-adic self-similar  strings, there is no need to distinguish between these various notions of `fractal dimensions'.

\begin{theorem}[Structure of the complex dimensions]\label{periodicity} 
Let $\mathcal{L}_p$ be a nontrivial $p$-adic self-similar   string. Then \\
\indent{\rm(}i{\rm)} The complex dimensions of $\mathcal{L}_p$ and the zeros of $\zeta_{\mathcal{L}_p}$ are periodically distributed along finitely many  vertical lines, with period $\mathbf p$, the oscillatory period of $\mathcal{L}_p$ {\rm(}as given in Definition \ref{oscillatoryperiod}{\rm)}. \\
\indent{\rm(}ii{\rm)}  Furthermore, along a given vertical line, each pole {\rm(}respectively, each zero{\rm)} of $\zeta_{\mathcal{L}_p}$ has the same multiplicity. \\
\indent{\rm(}iii{\rm)}  Finally,
 the dimension $D$ of $\mathcal{L}_p$ is the only complex dimension that is located on the real axis.
Moreover,
 $D$ is a simple pole of $\zeta_{\mathcal{L}_p}$
and is located on the right most vertical line.
That is,
$D$ is equal to the maximum of the real parts of the complex dimensions.  
\end{theorem}

\begin{remark} 
The situation described above---specifically, the rationality of the zeta function in the variable $z=r^{s}$, with $r=p^{-d}$, and the ensuing periodicity of the poles and the zeros---is analogous to the one encountered for a curve (or more generally, a variety) over a finite field $\mathbb{F}_{p^d}$; see, e.g., Chapter 3 of \cite{ParSh}. In this analogy, the prime number $p$ is the characteristic of the finite field, and the cardinality of the field, $p^d$, corresponds to $r^{-1}$, the reciprocal of the multiplicative generator of $\mathcal L_p$.
\end{remark}

We next supplement the above results by stating a theorem (from \cite{LapLu2}, \cite{LapLu4} and based on corresponding results in \cite[Chs. 2 \& 6]{L-vF2}) which  will be very useful to us in \S\ref{exact tf} in order to simplify the tube formula associated with a $p$-adic self-similar string.\footnote{In light of Theorems \ref{rationality} and \ref{periodicity}, Theorem \ref{residue} follows from a corresponding result in \cite{L-vF2}.} 

According to part (i) of Theorem \ref{periodicity}, there exist finitely many poles 
\[\omega_1,\ldots,  \omega_q,\]
of  $\zeta_{\mathcal{L}_p}$ with $\omega_1=D$ and $\Re(\omega_q)\leq \cdots \leq \Re(\omega_2)<D$,
such that 
\[\mathcal D_{\mathcal L_p}=\{\omega_u+in\mathbf{p} : n\in \mathbb Z, ~u=1, \ldots, q\}.\]
Furthermore, each complex dimension $\omega+in\mathbf{p}$ is simple (by parts (ii) and (iii) of Theorem \ref{periodicity}) and the residue of $\zeta_{\mathcal L_p}(s)$ at $s=\omega+in\mathbf{p}$ is independent of $n\in \mathbb Z$ and, in light of Equation (\ref{rational function}), equal to 
\begin{equation}\label{residue equation}
\res(\zeta_{\mathcal{L}_p}; \omega+in\mathbf{p})=
\lim_{s\to\omega}(s-\omega)
\frac
{\sum_{k =1}^{K} r^{m_{k}'s}}
{1-\sum_{j=1}^Nr^{n_j's}}=
\frac
{\sum_{k =1}^{K} r^{m_{k}'\omega}}
{\log{r^{-1}}\sum_{j=1}^N n_j'r^{n_j'\omega}}.
\end{equation}
In particular, this is the case for $\omega=D$.
See~\cite[Ch. 6]{L-vF2} for the general case.
 
\begin{theorem}\label{residue}
{\rm(}i{\rm)} For each $u=1, \ldots, q,$ the principal part of the Laurent series of $\zeta_{\mathcal L_p}(s)$ at 
$s=\omega_u+in\mathbf{p}$ does not depend on $n\in \mathbb Z$. \\
\indent{\rm(}ii{\rm)}  Moreover, let $u\in \{1, \ldots, q\}$ be such that $\omega_u$ {\rm(}and hence also $\omega_u+in\mathbf{p}$, for every $n\in \mathbb Z$, by part {\rm(}ii{\rm)}  of Theorem \ref{periodicity}{\rm)} is simple. 
Then the residue of $\zeta_{\mathcal L_p}(s)$ at $s=\omega_u+in\mathbf{p}$ is independent of  $n\in \mathbb Z$ and 
\begin{equation}\label{residue equation}
\res(\zeta_{\mathcal{L}_p}; \omega_u+in\mathbf{p})=
\frac
{\sum_{k =1}^{K} r^{m_{k}'\omega_u}}
{\log{r^{-1}}\sum_{j=1}^N n_j'r^{n_j'\omega_u}}.
\end{equation}
In particular, this is the case for $\omega_1=D.$ 
\end{theorem} 
 
 Note that by contrast, in the lattice case of the archimedean theory of self-similar strings developed in \cite[Chs.~2 and 3]{L-vF2}, one has to assume that the gap sizes (and not just the scaling ratios) are integral powers of $r$ in order to obtain the counterpart of Theorem \ref{residue}. 
 
\begin{remark}[Comparison with the archimedean case]\label{stronglattice}
Part (i) of Theorem \ref{periodicity}, along with Theorem \ref{rationality}, shows that  the theory of  $p$-adic self-similar  strings is simpler than its archimedean counterpart. Indeed, not only is it the case that every $p$-adic self-similar  string $\mathcal L_p$ is lattice, but both the zeros and poles of $\zeta_{\mathcal L_p}(s)$ are periodically distributed along vertical lines, with the same period (because $\mathcal L_p$ is strongly lattice; see Theorem \ref{lattice}). By contrast, even if an archimedean self-similar string $\mathcal L$ is assumed to be `lattice', then the zeros of $\zeta_{\mathcal L}(s)$ are usually not periodically distributed because the multiplicative group generated by the distinct gap sizes need not be of rank one or coincide with the group generated by the distinct scaling ratios; see  \cite[Chs.~2 and 3]{L-vF2}.
In fact, from this point of view, only strongly lattice archimedean (or real) strings behave like $p$-adic self-similar strings. 
\end{remark}

 \section{Exact Tube Formulas for  $p$-Adic  Self-Similar  Strings}\label{exact tf}
   In view of Equation (\ref{sszf}), every $p$-adic self-similar   string $\mathcal L_p$ is strongly languid, with $\kappa=0$ and $A=r_Ng_K^{-1},$ in the notation of \cite[Definition 5.3]{L-vF2}. Indeed, Equation (\ref{sszf}) implies that 
$|\zeta_{\mathcal L_p}(s)|\ll(r_N^{-1}g_K)^{-|\Re(s)|}$, as $\Re(s)\rightarrow -\infty.$ 
Hence, we can apply the distributional  tube formula without error term (i.e., the last part of Theorem 
\ref{dtf} and of Corollary \ref{ftf}) with $W=\mathbb C$. Since by Theorem \ref{lattice}, $\mathcal L_p$ is a lattice string, we obtain (in light of Theorems \ref{rationality}, \ref{periodicity} and   \ref{residue}) the following simpler analogue of Theorem 8.25 in \cite{L-vF2}:\footnote{We note that instead,
 we could more generally apply parts (i) and (ii) of Theorem \ref{dtf} in order to obtain a distributional tube formula with or without error term, valid \emph{without} assuming that all of the complex dimensions of $\mathcal L_p$ are simple. This observation is used in the proof of Theorem \ref{truncated}. }

\begin{theorem} [Exact tube formulas for $p$-adic self-similar fractal strings]\label{exact tube formula}
Let $\mathcal{L}_p$ be a $p$-adic self-similar  string with simple complex dimensions. Then, for all $\varepsilon$ with $0<\varepsilon<g_Kr_N^{-1}$, the volume $V_{\mathcal L_p}(\varepsilon)$ is given by
\begin{equation}\label{ssvolume1}
V_{\mathcal L_p}(\varepsilon)= \sum_{\omega \in \mathcal D} c_{\omega}\varepsilon^{1-\omega},
\end{equation}
where $c_{\omega}=\frac{\res(\zeta_{\mathcal{L}_p}; ~ \omega)}{p(1-\omega)}$ for each $\omega \in \mathcal D=\mathcal D_{\mathcal L_p}(\mathbb C).$
 \end{theorem}

\begin{corollary}\label{5.13}
Let $\mathcal{L}_p$ be a $p$-adic self-similar   string with multiplicative generator $r$. Assume that all the complex dimensions of  $\mathcal{L}_p$ are simple. Then, for all $\varepsilon$ with $0<\varepsilon<g_Kr_N^{-1}$, the volume $V_{\mathcal L_p}(\varepsilon)$ is given by the following exact distributional tube formula{\rm:}
\begin{equation}\label{ssvolume2}
V_{\mathcal L_p}(\varepsilon)= \sum_{u=1}^q \varepsilon^{1-\omega_u}G_u(\log_{1/r}\varepsilon^{-1}),
\end{equation}
where $1/r=p^d$ {\rm(}as in Equation {\rm(}\ref{Q}{\rm)}{\rm)}, and for each $u=1, \ldots, q,~ G_u$ is a real-valued periodic function of period 1 on $\mathbb R$ corresponding to the line of complex dimensions through 
$\omega_u ~ (\omega_1=D>\Re (\omega_2) \geq \cdots \geq \Re (\omega_q)),$ and is given by the following {\rm(}conditionally and also distributionally convergent{\rm)} Fourier series{\rm:} 

 \begin{equation}\label{Gu}
G_u(x)= \frac{\res(\zeta_{\mathcal{L}_p};\omega_u)}{p}
\sum_{n\in \mathbb Z} \frac{e^{2\pi i n x}}{ 1-\omega_u-in\mathbf{p} },
\end{equation}
where {\rm(}as in Equation {\rm(}\ref{residue equation}{\rm)} of Theorem \ref{residue}{\rm)}, 
 \begin{equation*}
\res(\zeta_{\mathcal{L}_p}; \omega_u )=
\frac
{\sum_{k =1}^{K} r^{m_{k}'\omega_u}}
{\log{r^{-1}}\sum_{j=1}^N n_j'r^{n_j'\omega_u}}.
\end{equation*}
Moreover, $G_u$ is nonconstant and bounded.
 \end{corollary}

 \begin{proof}
 That the explicit formula for $V_{\mathcal L_p}(\varepsilon)$ can be written as a sum over $\varepsilon^{1-\omega_u}$ times a periodic function of period $1$ in $\log_{1/r}\varepsilon^{-1}$ in case all complex dimensions are simple follows from Theorem~\ref{exact tube formula},
as does the formula for $G_u$.
This latter function is clearly nonconstant.
That it is bounded follows from~\cite[Formula~(1.13)]{L-vF2}.
 \end{proof}
 
 \begin{remark}\label{strong lattice}
In comparing our results with the corresponding results in Chapter 2 and \S8.4 of \cite{L-vF2}, obtained for real self-similar fractal strings, the reader should keep in mind the following two facts: (i) the simplification brought upon by the ``strong lattice property'' of $p$-adic self-similar   strings; see Theorem \ref{residue} and Remark \ref{stronglattice} above. 
(ii) By construction, any $p$-adic self-similar  string $\mathcal L_p$ (as defined in this paper) has total length $L$ equal to one: $L=\mu_H(\mathcal L_p)=\zeta_{\mathcal L_p}(1)=\mu_H(\mathbb Z_p)=1.$ Indeed, for notational simplicity, we have assumed that the similarity transformations 
$\Phi_j$ $(j=1, \dots, N)$ are self-maps of the `unit interval' $\mathbb Z_p$, rather than of an arbitrary `interval' of length $L$ in $\mathbb Q_p$. Clearly, only minor adjustments are needed in order to deal with the case of an arbitrary interval. 
\end{remark}

\begin{remark}
It would be interesting to obtain a geometric interpretation of the coefficients of the fractal tube formulas (\ref{ssvolume1}) and (\ref{ssvolume2}), in terms of nonarchimedean fractal curvatures, along the lines suggested by the work of \cite{L-vF2} and  \cite{LapPe1, LapPe2, LPW} (in the archimedean setting). 
It would also be interesting to extend these results to higher-dimensional $p$-adic self-similar sets or tilings (as was done in the Euclidean case in \emph{loc. cit.}). 
\end{remark}

\begin{theorem}[Truncated tube formula]\label{truncated}
Let $\mathcal{L}_p$ be an arbitrary $p$-adic self-similar   string with multiplicative generator $r$.  Then, for all $\varepsilon$ with $0<\varepsilon<g_Kr_N^{-1}$, 
\begin{equation}\label{truncated tf}
V_{\mathcal L_p}(\varepsilon)=\varepsilon^{1-D}(G(\log_{1/r}\varepsilon^{-1}) + o(1)), 
\end{equation}
where $o(1)\rightarrow 0$ as $\varepsilon \rightarrow 0^+$ and  $G=G_1$ is the nonconstant, bounded periodic function of period 1 given by Equation {\rm(}\ref{Gu}{\rm)} of Theorem \ref{5.13} {\rm(}with $u=1$ and $\omega_1=D${\rm)}.
\end{theorem}

\begin{proof}
This follows from the method of proof of 
Corollary 8.27 in  \cite{L-vF2}   in the easy case of a lattice string and with $2\varepsilon$ replaced by 
$\varepsilon$  and with $L:=1;$  see Remark \ref{strong lattice}.
In particular, we have the following `truncated tube formula':
\begin{equation}\label{4.25}
V_{\mathcal L_p}(\varepsilon)=\varepsilon^{1-D}G(\log_{1/r}\varepsilon^{-1}) + E(\varepsilon), 
\end{equation}
where $E(\varepsilon)$ is an error term that can be estimated much as in \emph{loc.~cit.}  In particular,   there exists $\delta>0$ such that $\varepsilon^{-(1-D)}E(\varepsilon)=O(\varepsilon^{\delta}),$  as $\varepsilon\rightarrow 0^+.$

Furthermore, since we limit ourselves here to the first line of complex dimensions, and since those complex dimensions are always simple (by parts (ii) and  (iii) of Theorem \ref{periodicity}), we do \emph{not}
have to assume (as in Theorem \ref{exact tube formula} and Corollary \ref{5.13}) that all the complex dimensions of $\mathcal L_p$ are simple in order for Equation (\ref{4.25}) and the corresponding error estimate for $E(\varepsilon)$ to be valid. 

More specifically, we note that Equation (\ref{4.25}) and the corresponding error estimate for 
$E(\varepsilon)$ (namely, $\delta>0$ and so $E(\varepsilon)=o(\varepsilon^{-(1-D)})$ as $\varepsilon \rightarrow 0^+$) follow from the first part of Theorem \ref{dtf} (the explicit tube formula with error term, applied to a suitable window), along with the fact that the complex dimensions on the rightmost vertical line $\Re(s)=D$ are simple (according to parts (ii) and (iii) of Theorem \ref{periodicity}).
Here, since $\mathcal L_p$ is a lattice string, we can simply choose the screen $S$ to be a vertical line lying strictly between $\Re(s)=D$ and the next vertical line of complex dimensions (if such a line exists). 
 \end{proof}

 \section{The Average Minkowski Content}\label{amc}  
 
The (inner) Minkowski dimension and the (inner) Minkowski content of a $p$-adic fractal string $\mathcal L_p$ (or, equivalently, of its metric boundary $\beta \mathcal L_p$, see Definition
  \ref{volume definition}) are defined exactly as the corresponding notion for a real fractal string 
  (see \cite{L-vF2}, Definition 1.2), 
 except for the fact that we use the definition of $V(\varepsilon)=V_{\mathcal L_p}(\varepsilon)$ provided in Equation (\ref{volume}) of \S\ref{inner tube}.
More specifically,
 the \emph{Minkowski dimension} of $\mathcal L_p$ is given by 
\begin{equation}\label{dimension}
D_M:= \inf  \left\{\alpha \geq  0: V_{\mathcal L_p}(\varepsilon)=O(\varepsilon ^{1-\alpha}) ~\mbox{as}~
\varepsilon \rightarrow 0^{+}  \right\}.  \end{equation}
Furthermore, $\mathcal L_p$ is said to be \emph{Minkowski measurable}, with \emph{Minkowski content}
$\mathcal M$, if the limit 
\begin{equation}
\mathcal M=\lim_{\varepsilon\rightarrow 0^{+}} V_{\mathcal L_p}(\varepsilon)\varepsilon ^{-(1-D_M)}
\end{equation}
exists in  $(0, \infty).$ Otherwise, $\mathcal L_p$ is said to be \emph{Minkowski nonmeasurable.} 

\begin{remark}
Note that since 
$V_{\mathcal L_p}(\varepsilon)=\mathcal V_{\mathcal L_p}(\varepsilon)-\mu_H(\beta \mathcal L_p)$
in light of Equation (\ref{volume}),
there is an analogy between the above definition of the Minkowski dimension and that of ``exterior dimension'', which is used in chaos theory to study certain archimedean `fat fractals' (dynamically defined fractals with positive Lebesgue measure); see, e.g., \cite{GMOY} and the survey article \cite{Ott}.
In the present nonarchimedean case, however, for any $p$-adic fractal string, it is necessary to substract $\mu_H(\beta \mathcal L_p)$ from $\mathcal V_{\mathcal L_p}(\varepsilon)$. 
Indeed, otherwise, the metric boundary of every $p$-adic string (even a single interval) would be a `fat fractal'; see \cite{LapLu3} and Remark \ref{limit} above.
\end{remark}
 
 The next result follows  from the truncated tube formula provided in Theorem \ref{truncated}, along with the corresponding error estimate. 

\begin{theorem}\label{Minkowski nonmeasurable}
A  $p$-adic self-similar  string $\mathcal L_p$ is never Minkowski measurable.  Moreover, it has multiplicatively periodic oscillations of order $D$ in its geometry. 
\end{theorem}

\begin{proof}
This follows immediately from Theorem \ref{truncated}
and the fact that $G=G_1$ is a \emph{nonconstant} periodic function, which implies (in light of Equation (\ref{truncated tf})) that the limit of 
$\varepsilon^{-(1-D)}V_{\mathcal L_p}(\varepsilon)$ does not exist  as $\varepsilon \rightarrow 0^+$.
\end{proof}
 
According to Theorem \ref{Minkowski nonmeasurable}, a $p$-adic  self-similar string does not have a well-defined Minkowski content, because it is not Minkowski measurable. Nevertheless, as we shall see in Theorem \ref{average content} below, it does have a suitable `average content' $\mathcal{M}_{av},$ in the following sense:

\begin{definition}\label{defcontent}
Let $\mathcal L_p$ be a $p$-adic fractal string of dimension $D$. The \emph{average Minkowski content}, $\mathcal{M}_{av}$, is defined by the logarithmic Cesaro average
\[
\mathcal{M}_{av}=\mathcal{M}_{av}(\mathcal L_p):=\lim_{T\rightarrow \infty}\frac{1}{\log T}\int_{1/T}^1\varepsilon^{-(1-D)}V_{\mathcal L_p}(\varepsilon) 
\frac{d\varepsilon}{\varepsilon},
\]
provided this limit exists and is a finite positive real number. 
\end{definition}

\begin{theorem}\label{average content} 
Let $\mathcal L_p$ be a $p$-adic self-similar  string of dimension $D$. Then the average Minkowski content of  $\mathcal {L}_p$ exists and is given by the finite positive number
\begin{equation}\label{content}
\mathcal{M}_{av}=\frac{1}{p(1-D)} \res(\zeta_{\mathcal L_p}; D)=\frac{1}{p(1-D)} \frac
{\sum_{k =1}^{K} r^{m_{k}'D}}
{\log{r^{-1}}\sum_{j=1}^N n_j'r^{n_j'D}}.
\end{equation} 
\end{theorem}

 \begin{proof}
In light of (the proof of) Theorem \ref{truncated},  we have for all $0<\varepsilon\leq 1$ and for some $\delta>0$,
\[
\varepsilon^{-(1-D)}V(\varepsilon)=G(\log_{1/r}\varepsilon^{-1}) + O(\varepsilon^{\delta}),
\]
where $G$ is the nonconstant and bounded periodic function of period~$1$ given by Equation~\eqref{Gu}
of Theorem \ref{5.13} (with $u=1$ and $\omega_1=D$). (See Equation (\ref{4.25}) and the text surrounding it.)
Noting that
\[
\lim_{T\to\infty}\frac{1}{\log T}\int_{1/T}^1\varepsilon^\delta
\,\frac{d\varepsilon}{\varepsilon}=0,
\]
and that each oscillatory term of $G_1$ (for $n\ne0$ in~\eqref{Gu}, $n\in \mathbb Z$) gives a vanishing contribution as well,
\[
\lim_{T\to\infty}\frac{1}{\log T}\int_{1/T}^1\varepsilon^{in/\log r}
\,\frac{d\varepsilon}{\varepsilon}=0,
\]
we conclude that
\[
\lim_{T\to\infty}\frac{1}{\log T}\int_{1/T}^1\varepsilon^{-(1-D)}V_{\mathcal L_p}(\varepsilon)\frac{d\varepsilon}{\varepsilon}\]
gives the constant coefficient of $G=G_1$.
 \end{proof}
 
 \begin{remark}
Definition \ref{defcontent} and Theorem \ref{average content} are the exact nonarchimedean counterpart of \cite{L-vF2}, Definition 8.29 and Theorem 8.30.
\end{remark}

\begin{example}[Nonarchimedean Cantor string]\label{average content CS3}
The average Minkowski content of the nonarchimedean Cantor string $\mathcal{CS}_3$ is given by 
\[
\mathcal M_{av}(\mathcal{CS}_3)=\frac{1}{6(\log3-\log2)}.
\]
Indeed, we have seen in Example \ref{cantor string}  that $ D=\log_32,  \res(\zeta_{\mathcal {CS}_3}; D)=1/2\log3$ and $p=3$.
\end{example}

\begin{example}[Nonarchimedean Fibonacci string] \label{fibonacci}
Let $\Phi_1~ \mbox{and}~  \Phi_2$ be the two affine similarity contraction mappings of $\mathbb {Z}_2$ given (much as in \S\ref{sss}, with $N=p=2$) by 
\begin{equation*}
\Phi_1(x)=2x \quad \mbox{and} \quad \Phi_2(x)=1+4x,
\end{equation*}
with the respective scaling ratios $r_1=1/2$ and $r_2=1/4$. 
The associated $2$-adic self-similar string (introduced in \cite{LapLu2}) with generator $G=3+4\mathbb Z_2$ is called the 
\emph{nonarchimedean} \emph{Fibonacci string} and  denoted by $\mathcal{FS}_2$ (compare with the archimedean counterpart discussed in \cite[\S 2.3.2]{L-vF2}). 
It is given by the sequence
$\mathcal{FS}_2 =l_1, l_2, l_3, \dots$ and consists (for $m=1,2, \ldots$) of intervals of lengths 
$l_m=2^{-(m+1)}$ with multiplicities $f_{m},$ the Fibonacci numbers. (Recall that these numbers are defined by the recursive formula: $f_{m+1}=f_m+f_{m-1}, f_0=0$ and $f_1=1$.)
Alternatively, in the spirit  of Theorem \ref{complementary}, 
the nonarchimedean Fibonacci string is the bounded open subset of $\mathbb Z_2$  given by the following disjoint union of 2-adic intervals (necessarily its 2-adic convex components): 
\[\mathcal{FS}_2=(3+4\mathbb{Z}_2)\cup (6+8\mathbb{Z}_2)\cup (12+16\mathbb{Z}_2)\cup (13+
16\mathbb{Z}_2) \cup \cdots.\]
By Theorem \ref{geometric zeta function}, the geometric zeta function of $\mathcal{FS}_2$ is given (almost exactly as for the archimedean Fibonacci string, cf.~\emph{loc.~cit.}) by\footnote
{The minor difference between the two geometric zeta functions is due to the fact that the real Fibonacci string $\mathcal{FS}$ in \cite[\S2.3.2 and Exple.~8.32]{L-vF2} has total length 4 whereas the present 2-adic Fibonacci string $\mathcal{FS}_2$
has total length 1; see also part (ii) of Remark \ref{strong lattice} above.}

 \begin{equation}\label{gzff}
\zeta_{\mathcal{FS}_2}(s)
=\frac{4^{-s}}{1-2^{-s}-4^{-s}}. 
\end{equation}
Hence, the set of complex dimensions  of $\mathcal {FS}_2$ is given by
\begin{equation}\label{cdfs2}
\mathcal{D}_{\mathcal {FS}_2}=\{   D+i n \mathbf{p}:n \in \mathbb{Z} \}\cup 
\{  -D+i (n+1/2) \mathbf{p}:n \in \mathbb{Z}\}
\end{equation}
with $D=\log_2 \phi$, where $\phi=(1+\sqrt5)/2$ is the golden ratio,  and $\mathbf{p}=2\pi / \log2,$ the oscillatory period of $\mathcal{FS}_2$; see Figure 4. 
Moreover, a simple computation shows that 
\begin{equation}\label{residue1}
\res(\zeta_{\mathcal{FS}_2}; D+i n \mathbf{p})=\frac{3-\phi}{5\log2}
\end{equation}
and 
\begin{equation}\label{residue2}
\res(\zeta_{\mathcal{FS}_2}; -D+i(n+1/2) \mathbf{p})=\frac{2+\phi}{5\log2},
\end{equation}
independently of $n\in \mathbb Z$.

We refer the interested reader to \cite{LapLu2} for additional information concerning the nonarchimedean Fibonacci string. 
\begin{figure}[h]\label{fib figure}
\psfrag{p}{$\mathbf{p}$}
\psfrag{.5p}{$\frac{1}{2}\mathbf{p}$}
\psfrag{-1}{$-1$}
\psfrag{-D}{$-D$}
\psfrag{0}{$0$}
\psfrag{D}{$D$}
\psfrag{1}{$1$}
\raisebox{-1cm}
{\psfig{figure=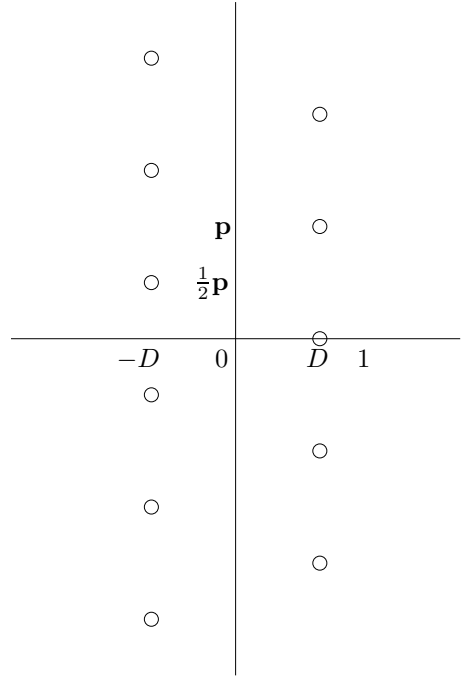, height=9 cm}}
\caption{The complex dimensions of the 2-adic Fibonacci string $\mathcal{FS}_2$. Here, $D=\log_2 \phi$ and $\mathbf{p}=2\pi / \log2.$}
\end{figure}

Note that $\zeta_{\mathcal{FS}_2}$ does not have any zero (in the variable $s$) since the equation $4^{-s}=0$ does not have any complex solution.
Moreover, in agreement with Theorem \ref{rationality}, $\zeta_{\mathcal{FS}_2}$ is a rational function of $z=2^{-s},$ i.e.,  
\begin{equation}\label{zfs2}
\zeta_{\mathcal{FS}_2}(s)=\frac{z^2}{1-z-z^2}. 
\end{equation}
Since, in light of  (\ref{zfs2}),
 the complex dimensions of $\mathcal{FS}_2$ are simple, we may apply either  Corollary \ref{ftf} or Corollary \ref{5.13}  in order to obtain the following exact fractal tube formula 
for the nonarchimedean Fibonacci string:\footnote
{In light of Theorem \ref{volume}, one can also directly derive this formula for 
$V_{\mathcal {FS}_2}(\varepsilon)$, much as was done for $V_{\mathcal{CS}_3}(\varepsilon)$ in Example~\ref{cantor volume}, although with some more strenuous work.}
  \begin{eqnarray}\label{fib vol}
V_{\mathcal {FS}_2}(\varepsilon)&=& \frac{1}{2}
 \sum_{\omega\in \mathcal D_{\mathcal{FS}_2}}\res(\zeta_{\mathcal {FS}_2};\omega)
  \frac{\varepsilon^{1-\omega}}{1-\omega}\\
   &=&
   \varepsilon^{1-D}G_1(\log_2\varepsilon^{-1})+
      \varepsilon^{1+D-i\mathbf{p}/2}G_2(\log_2\varepsilon^{-1})\nonumber, 
                          \end{eqnarray}
                where, in light Equation (\ref{cdfs2}) and of the values of $\res(\zeta_{\mathcal{FS}_2};\omega)$ provided in Equations (\ref{residue1}) and (\ref{residue2}), 
                 $G_1$ and $G_2$ are bounded periodic functions of period 1 on $\mathbb R$ given by their respective (conditionally convergent) Fourier series
                \begin{equation}
                G_1(x)= \frac{3-\phi}{10\log 2}
 \sum_{n\in \mathbb Z}
   \frac{e^{2\pi i n x}}{1-D-in\mathbf p} 
                   \end{equation}
                   and 
                    \begin{equation}
                G_2(x)= \frac{2+\phi}{10\log 2}
                 \sum_{n\in \mathbb Z}
   \frac{e^{2\pi i n x}}
  {1+D-i(n+1/2)\mathbf p}.
                     \end{equation}
Note that the above Fourier series for $G_1$ and $G_2$ are conditionally (and also distributionally) convergent, for all $x\in\mathbb R.$
 Furthermore,
the explicit fractal tube formula (\ref{fib vol}) for $\mathcal{FS}_2$ actually holds pointwise and not just distributionally, as the interested reader may verify via a direct computation.
The average Minkowski content of  $\mathcal{FS}_2$ is given by 
\[
\mathcal M_{av}=\mathcal M_{av}(\mathcal{FS}_2)=\frac{1}{2(\phi +2)(\log2- \log \phi)},
\]
where $\phi=\frac{1+ \sqrt 5}{2}$ is the golden ratio.
Indeed, since  $D=\log_2 \phi $, we deduce from Equation (\ref{residue1}) with $n=0$ that 
\[\res(\zeta_{\mathcal {FS}_2}; D)=\frac{1}{(\phi+2)\log2}.\]
Hence, the above expression for $\mathcal M_{av}$ 
follows from Theorem \ref{average content} with $p=2$. Furthermore, note that 
$\log2- \log \phi=\log(\sqrt 5-1).$
Therefore, $\mathcal M_{av}$ can be rewritten  as follows:
 \[
\mathcal M_{av}=\frac{1}{(5+\sqrt{5})\log(\sqrt 5-1)}.
\]
\end{example}

\bibliographystyle{amsplain}

\end{document}